\documentclass[12pt,a4paper,english,reqno]{amsart}
\usepackage[a4paper,footskip=1.5em]{geometry}
\usepackage{amsmath,amssymb,amsthm,mathtools}
\usepackage[mathscr]{euscript}
\usepackage{adjustbox,tikz,calc,graphics,babel}
\usepackage{csquotes,enumerate,verbatim}
\usepackage[final]{microtype}
\usepackage[numbers]{natbib}
\usetikzlibrary{shapes.misc,calc,intersections,patterns,decorations.pathreplacing}
\usepackage{hyperref}
\hypersetup{colorlinks=true,linkcolor=blue,citecolor=blue}

\theoremstyle{plain}
\newtheorem*{theorem*}{Theorem}
\newtheorem{theorem}{Theorem}[section]
\newtheorem{lemma}[theorem]{Lemma}
\newtheorem{claim}[theorem]{Claim}
\newtheorem{proposition}[theorem]{Proposition}
\newtheorem*{claim*}{Claim}

\theoremstyle{remark}

\def\N{\mathbb{N}}

\def\R{\mathbb{R}}
\def\P{\mathbb{P}}
\def\E{\mathbb{E}}
\def\C{\mathcal}

\def\Scr{\mathscr}

\DeclareMathOperator\V{Var}
\let\emptyset\varnothing
\let\eps\varepsilon

\let\originalleft\left
\let\originalright\right
\renewcommand{\left}{\mathopen{}\mathclose\bgroup\originalleft}
\renewcommand{\right}{\aftergroup\egroup\originalright}

\makeatletter
\def\imod#1{\allowbreak\mkern10mu({\operator@font mod}\,\,#1)}
\makeatother

\def\NR{[n]^{(r)}}
\def\VV{\mathbf{V}}
\def\NN{\mathbf{N}}
\def\MM{\mathbf{M}}
\def\RR{\mathbf{R}}

\begin{document}

\title{Transference for the Erd\H{o}s--Ko--Rado theorem}

\author{J\'{o}zsef Balogh}
\address{Department of Mathematics, University of Illinois, 1409 W.\/ Green Street, Urbana IL 61801, USA; {\em and\/}
Bolyai Institute, University of Szeged, 6720 Szeged, Hungary}
\email{jobal@math.uiuc.edu}

\author{B\'{e}la Bollob\'{a}s}
\address{Department of Pure Mathematics and Mathematical Statistics, University of Cambridge, Wilberforce Road, Cambridge CB3\thinspace0WB, UK; {\em and\/}
Department of Mathematical Sciences, University of Memphis, Memphis TN 38152, USA; {\em and\/} London Institute for Mathematical Sciences, 35a South St., Mayfair, London W1K\thinspace2XF, UK.}
\email{b.bollobas@dpmms.cam.ac.uk}

\author{Bhargav Narayanan}
\address{Department of Pure Mathematics and Mathematical Statistics, University of Cambridge, Wilberforce Road, Cambridge CB3\thinspace0WB, UK}
\email{b.p.narayanan@dpmms.cam.ac.uk}

\date{12 October 2014}

\begin{abstract}
For natural numbers $n,r \in \N$ with $n\ge r$, the Kneser graph $K(n,r)$ is the graph on the family of $r$-element subsets of $\{1,\dots,n\}$ in which two sets are adjacent if and only if they are disjoint. Delete the edges of $K(n,r)$ with some probability, independently of each other: is the independence number of this random graph equal to the independence number of the Kneser graph itself? We answer this question affirmatively as long as $r/n$ is bounded away from $1/2$, even when the probability of retaining an edge of the Kneser graph is quite small. This gives us a random analogue of the Erd\H{o}s--Ko--Rado theorem since an independent set in the Kneser graph is the same as a uniform intersecting family. To prove our main result, we give some new estimates for the number of disjoint pairs in a family in terms of its distance from an intersecting family; these might be of independent interest.
\end{abstract}

\maketitle

\section{Introduction}
Over the past twenty years, a great deal of work has gone into proving `sparse random' analogues of classical extremal results in combinatorics. Some of the early highlights include a version of Mantel's theorem for random graphs proved by Babai, Simonovits, and Spencer~\citep{Triangle1}, the Ramsey theoretic results of R{\"o}dl and Ruci{\'n}ski~\citep{Ramsey1, Ramsey2}, and a random analogue of Szemer\'edi's theorem due to Kohayakawa, {\L}uczak and R{\"o}dl~\citep{Roth2}. Very general transference theorems have since been proved by Conlon and Gowers~\citep{Conlon}, Schacht~\citep{Schacht}, Balogh, Morris and Samotij~\citep{cont1} and Saxton and Thomason~\citep{cont2}. The surveys of {\L}uczak~\citep{survey2} and R\"{o}dl and Schacht~\citep{survey1} provide a detailed account of such results.

In this paper, we shall be interested in proving such a transference result for a central result in extremal set theory, the Erd\H{o}s--Ko--Rado theorem. A family of sets $\C{A}$ is said to be \emph{intersecting} if $A \cap B \neq \emptyset$ for all $A, B \in \C{A}$. Writing $X^{(r)}$ for the family of all $r$-element subsets of a set $X$ and $[n]$ for the set $\{ 1,2,...,n\}$, a classical result of Erd\H{o}s, Ko and Rado~\citep{EKR} asserts that if $n > 2r$ and $\C{A} \subset \NR$ is intersecting, then $|\C{A}| \le \binom{n-1}{r-1}$ with equality if and only if $\C{A}$ is a \emph{star}. As is customary, we define the \emph{star centred at $x\in[n]$} to be the family of all the $r$-element subsets of $[n]$ containing $x$ and we call an intersecting family \emph{trivial} if it is contained in a star.

If $\C{A}\subset \NR$ is intersecting and has cardinality comparable to that of a star, must $\C{A}$ necessarily resemble a star? Such questions about the `stability' of the Erd\H{o}s--Ko--Rado theorem have received a great deal of attention. Perhaps the earliest stability result about the Erd\H{o}s--Ko--Rado theorem was proved by Hilton and Milner~\citep{Hilton} who determined how large a uniform intersecting family can be if one insists that the family is nontrivial. Furthering this line of research, Friedgut~\citep{Stab1}, Dinur and Friedgut~\citep{Stab2}, and Keevash and Mubayi~\citep{Stab3} have shown that every `large' uniform intersecting family is essentially trivial. Finally, let us mention that Balogh, Das, Delcourt, Liu and Sharifzadeh~\citep{typical} have recently shown, amongst other things, that almost all $r$-uniform intersecting families are trivial when $r < (n - 8\log n)/3$.

As stated earlier, our aim in this note to prove a transference result for the Erd\H{o}s--Ko--Rado theorem. The notion of stability we shall consider here was introduced by Bollob\'as, Narayanan and Raigorodskii~\citep{bnr} (see also~\citep{distance2}). To present this notion of stability, it will be helpful to consider $\NR$ in a different incarnation, as the {\em Kneser graph $K(n,r)$}. The Kneser graph $K(n,r)$ is the graph on $\NR$ where two vertices, i.e., $r$-element subsets of $[n]$, are adjacent if and only if they are disjoint. We shall freely switch between these two incarnations of $\NR$.
 
Observe that a family $\C{A} \subset \NR$ is intersecting if and only if $\C{A}$ induces an independent set in $K(n,r)$. Writing $\alpha (G)$ for the size of the largest independent set of a graph $G$, the Erd\H{o}s--Ko--Rado theorem asserts that $\alpha (K(n,r)) = \binom{n-1}{r-1}$ when $n > 2r$. Let $K_p(n,r)$ denote the random subgraph of $K(n,r)$ obtained by retaining each edge of $K(n,r)$ independently with probability $p$. Bollob\'as, Narayanan and Raigorodskii~\citep{bnr} asked the following natural question: is $\alpha(K_p(n,r)) = \binom{n-1}{r-1}$? They proved, when $r = r(n)= o(n^{1/3})$, that the answer to this question is in the affirmative even after practically all the edges of the Kneser graph have been deleted. More precisely, they showed that in this range, there exists a (very small) critical probability $p_c(n,r)$ with the following property: as $n \to \infty$, if $p/p_c > 1$, then with high probability, $\alpha(K_p(n,r)) = \binom{n-1}{r-1}$ and the only independent sets of this size in $K_p(n,r)$ are stars, while if $p/p_c<1$, then $\alpha(K_p(n,r)) > \binom{n-1}{r-1}$ with high probability.

Bollob\'as, Narayanan and Raigorodskii also asked what happens for larger values of $r$, and conjectured in particular that as long as $r/n$ is bounded away from $1/2$, such a random analogue of the Erd\H{o}s--Ko--Rado theorem should continue to hold for $K_p(n,r)$ for some $p$ bounded away from $1$. In this note, we shall prove this conjecture and a bit more.

\begin{theorem}\label{t:main}
For every $\eps>0$, there exist constants $c=c(\eps) > 0$ and $c'=c'(\eps)>0$ with $c < c'$ such that for all $n,r \in \N$ with $r \le (1/2 - \eps)n$,
\[
\P \left(\alpha (K_p(n,r)) = \binom{n-1}{r-1} \right) \to
\begin{cases}
   1 &\mbox{if } p \ge \binom{n-1}{r-1}^{-c}\\
   0 &\mbox{if } p \le \binom{n-1}{r-1}^{-c'}\\
   \end{cases}
\]
as $n \to \infty$. In particular, with high probability, $\alpha(K_{1/2}(n,r)) = \binom{n-1}{r-1}$.
\end{theorem}

All the work in proving Theorem~\ref{t:main} is in showing that $c(\eps)$ exists; as we shall see, the existence of $c'(\eps)$ follows from a simple second moment calculation.

Let us briefly describe some of the ideas that go into the proof of Theorem~\ref{t:main}. We shall prove two results which, taken together, show that a large family $\C{A} \subset \NR$ without a large intersecting subfamily must necessarily contain many pairs of disjoint sets, or in other words, must induce many edges in $K(n,r)$; we do this in Section~\ref{s:disj}. We put together the pieces and give the proof of Theorem~\ref{t:main} in Section~\ref{s:proof}. In Section~\ref{s:cont}, we briefly describe some approaches to improving the dependence of $c(\eps)$ on $\eps$ in Theorem~\ref{t:main}. We conclude with some discussion in Section~\ref{s:conc}.

\section{Preliminaries}\label{s:prelim}
Henceforth, a `family' will be a uniform family on $[n]$ unless we specify otherwise. To ease the notational burden, we adopt the following notational convention: when $n$ and $r$ are clear from the context, we write $\VV = \binom{n}{r}$, $\NN = \binom{n-1}{r-1}$, $\MM = \binom{n-r-1}{r-1}$ and $\RR = \binom{2r}{r}$. 

We need a few results from extremal set theory, some classical and some more recent. The first result that we will need, due to Hilton and Milner~\citep{Hilton}, bounds the cardinality of a nontrivial uniform intersecting family. Writing $\C{A}_x$ for the subfamily of a family $\C{A}$ that consists of those sets containing $x$, we have the following.

\begin{theorem}\label{HM}
Let $n, r \in \N$ and suppose that $n > 2r$. If $\C{A} \subset \NR$ is an intersecting family with $|\C{A}| \ge \NN - \MM +2$, then there exists an $x \in [n]$ such that $\C{A} = \C{A}_x$. \qed
\end{theorem}

The next result we shall require, due to Friedgut~\citep{Stab1}, is a quantitative extension of the Hilton--Milner theorem  which says that any sufficiently large uniform intersecting family must resemble a star.

\begin{theorem}\label{friedgut}
For every $\eps > 0$, there exists a $C = C(\eps) >0$ such that for all $n, r \in \N$ with $\eps n \le r \le (1/2 - \eps)n$, the following holds: if $\C{A} \subset \NR$ is an intersecting family and $|\C{A}| = \NN - k$, then there exists an $x \in [n]$ for which $|\C{A}_x| \ge \NN - Ck$. \qed
\end{theorem}

We will also need the following well-known inequality for cross-intersecting families due to the second author~\citep{setpairs}.

\begin{theorem}\label{BB-thm}
Let $(A_1, B_1), \dots , (A_m, B_m)$ be pairs of disjoint $r$-element sets such that $A_i \cap B_j \neq \emptyset$ for $i, j \in [m]$ whenever $i \neq j$. Then $m \le \RR $. \qed
\end{theorem}

Finally, we shall require a theorem of Kruskal~\citep{Kruskal} and Katona~\citep{Katona}. For a family $\C{A} \subset \NR$, its shadow in $[n]^{(k)}$, denoted $\partial^{(k)} \C{A}$, is the family of those $k$-sets contained in some member of $\C{A}$. For $x \in \R$ and $r \in \N$, we define the \emph{generalised binomial coefficient $\binom{x}{r}$} by setting
\[\binom{x}{r} = \frac{x(x-1)\dots(x-r+1)} {r!}.\]
The following convenient formulation of the Kruskal--Katona theorem is due to Lov\'asz~\citep{Lovasz}.

\begin{theorem}\label{lov}
Let $n,r,k\in \N$ and suppose that $k \le r \le n$. If the cardinality of $\C{A} \subset \NR$ is $\binom{x}{r}$ for some real number $x \ge r$, then $|\partial^{(k)} \C{A}| \ge \binom{x}{k}$. \qed
\end{theorem}

To avoid clutter, we omit floors and ceilings when they are not crucial. We use the standard $o(1)$ notation to denote any function that tends to zero as $n$ tends to infinity; the variable tending to infinity will always be $n$ unless we explicitly specify otherwise.

\section{The number of disjoint pairs}\label{s:disj}
Given a family $\C{A}$, we write $e(\C{A})$ for the number of disjoint pairs of sets in $\C{A}$; equivalently, $e(\C{A})$ is the number of edges in the subgraph of the Kneser graph induced by $\C{A}$. In this section, we give some bounds for $e(\C{A})$.

We denote by $\C{A}^*$ the largest intersecting subfamily of a family $\C{A}$; if this subfamily is not unique, we take any subfamily of maximum cardinality. We write $\ell(\C{A}) = |\C{A}| - |\C{A}^*|$ for the difference between the cardinality of $\C{A}$ and the largest intersecting subfamily of $\C{A}$. 

Trivially, we have $e(\C{A}) \ge \ell(\C{A})$. Our first lemma says that we can do much better than this trivial bound when $\ell(\C{A})$ is large.

\begin{lemma}\label{set-pairs} Let $n, r \in \N$. For any $\C{A} \subset \NR$,
\[ e(\C{A}) \ge \frac{\ell(\C{A})^2}{2\RR}.
\]
\end{lemma}
\begin{proof}
To prove this lemma, we need the notion of an induced matching. An {\em induced matching of size $m$} in a graph $G$ is a set of $2m$ vertices inducing a subgraph consisting of $m$ independent edges; equivalently, we refer to these $m$ edges as an induced matching of size $m$. The {\em induced-matching number} of $G$, in notation, $m(G)$, is the maximal size of an induced matching in $G$.

\begin{proposition}\label{new-set-pairs}
Let $G=(V,E)$ be a graph with $m(G)=m\ge 1$. Then 
\[|E|\ge \frac{k^2}{4m},\] 
where $k = |V| - \alpha(G)$.
\end{proposition}
\begin{proof}
Let us choose $X=\{x_1, \dots , x_m\}$ and $Y=\{y_1, \dots , y_m\}$ so that the edges $x_1y_1, \dots , x_my_m$ constitute an induced matching. Let $Z=\Gamma(X\cup Y)$ be the set of neighbours of the vertices in $X\cup Y$; thus $X\cup Y \subset Z$. Since $m(G)=m$, the set $V(G)\setminus Z$ is independent, so $|Z|\ge k$. Since some vertex in $X \cup Y$ has at least $|Z|/2m$ neighbours, we conclude that $\Delta(G)\ge |Z|/2m\ge k/2m$ where $\Delta(G)$ is the maximum degree of $G$.

Now define a sequence of graphs $G = G_0\supset G_1 \supset \dots \supset G_{k}$ and a sequence of vertices $x_0, x_1, \dots , x_{k}$ by taking $x_i$ to be a vertex of $G_i$ of maximal degree and $G_{i+1}$ to be the graph obtained from $G_i$ by deleting $x_i$. We know from our earlier arguments that $\Delta(G_i)\ge (k-i)/2m$, so $|E|\ge \sum_{i=0}^{k} \Delta(G_i)\ge k^2/4m$.
\end{proof}

To apply the previous proposition, we need the following corollary of Theorem~\ref{BB-thm} the proof of which is implicit in~\citep{typical}; we include the short proof here for completeness.

\begin{proposition}\label{BB-cor}
For $n\ge 2r$, the induced-matching number of $K(n,r)$ is 
\[m(K(n, r))=\binom{2r-1}{r-1}= \frac{\RR}{2}.\]
\end{proposition}

\begin{proof}
Let $A_1B_1, \dots , A_mB_m$ be an induced matching in $K(n, r)$. For $m+1\le i \le 2m$, we set $A_i=B_{i-m}$ and
$B_i=A_{i-m}$. We apply Theorem~\ref{BB-thm} to the pairs $(A_1, B_1), \dots , (A_{2m}, B_{2m})$ and conclude that $2m\le \RR$.

The $\RR/2$ partitions of $[2r]$ into disjoint $r$-sets form an induced matching, so $m(K(n, r))=\RR/2$, as claimed.
\end{proof}

The lemma follows by applying Proposition~\ref{new-set-pairs} to $G_{\C{A}}$, the subgraph of the Kneser graph $K(n,r)$ induced by $\C{A}$.
\end{proof}

Note that Lemma~\ref{set-pairs} is only effective when $\ell(\C{A}) \ge 2\RR$. The next, somewhat technical, lemma complements Lemma~\ref{set-pairs} by giving a better bound when $\ell(\C{A})$ is small provided the size of $\C{A}$ is large.

\begin{lemma}\label{kk}
For every $\eps, \eta > 0$, there exist constants $\delta = \delta(\eps, \eta) > 0$ and $C = C(\eps)>0$ with the following property: for all $n, r \in \N$ with $\eps n \le r \le (1/2 - \eps)n$, we have 
\[e(\C{A}) \ge \ell(\C{A})^{1+\delta} - C\ell(\C{A}) \] 
for any family $\C{A} \subset \NR$ with $|\C{A}| = \NN$ and $\ell(\C{A}) \le \NN^{1-\eta}$.
\end{lemma}

To clarify, the $C(\eps)$ in the statement of the lemma above is the same as the $C(\eps)$ guaranteed by Theorem~\ref{friedgut}.

\begin{proof}[Proof of Lemma~\ref{kk}]

First, let us note that since we always have $e(\C{A}) \ge \ell(\C{A})$, it suffices to prove the lemma under the assumption that $n$ is sufficiently large.

Let $\ell = \ell(\C{A})$. We start by observing that most of $\C{A}$ must be contained in a star. Indeed, as before, let $\C{A}^*$ denote the largest intersecting subfamily of $\C{A}$; by definition, $|\C{A}^*| = \NN - \ell$. Since we have assumed that $\eps n \le r \le (1/2-\eps) n$, we may assume, by Theorem~\ref{friedgut}, that $|\C{A}^*_n| \ge \NN - C\ell$, where $C = C(\eps)$ is as guaranteed by Theorem~\ref{friedgut}. Hence, $|\C{A}_n| \ge |\C{A}^*_n| \ge \NN -C\ell$. 

We also know that $|\C{A}_n| \le |\C{A}^*| \le \NN - \ell$; let $\C{B}$ be a subset of $\C{A} \setminus \C{A}_n$ of cardinality exactly $\ell$. We shall bound $e(\C{A})$ by counting the number of edges between $\C{B}$ and $\C{A}_n$ in $K(n,r)$. 

Let us define
\[\C{A}' = \left\{ A \setminus \{n\}: A \in \C{A}_n \right\} \subset [n-1]^{(r-1)}\] 
and
\[\C{B}' = \left\{ [n-1] \setminus B: B\in \C{B} \right\} \subset [n-1]^{(n-r-1)}.\] 
Clearly, to count the number of edges between $\C{A}_n$ and $\C{B}$ in $K(n,r)$, it suffices to count the number of pairs $(A', B')$ in $\C{A}' \times \C{B}'$ with $A' \subset B'$. This quantity is obviously bounded below by the number of sets $A' \in \C{A}'$ contained in at least one $B' \in \C{B}'$.

Since $\C{A}' \subset [n-1]^{(r-1)}$ and $|\C{A}'| \ge \NN - C\ell$, the number of sets $A' \in \C{A}'$ contained in some $B' \in \C{B}'$ is at least $|\partial^{(r-1)}\C{B}'| - C\ell$. Consequently,
\[ e(\C{A}) \ge |\partial^{(r-1)}\C{B}'| - C\ell.\]
We shall show that there exists a $\delta = \delta(\eps, \eta) >0$ such that, under the conditions of the lemma, $|\partial^{(r-1)}\C{B}'| \ge \ell ^{1+\delta}$ for all sufficiently large $n \in \N$. We deduce the existence of such a $\delta$ from Theorem~\ref{lov}, the Kruskal--Katona theorem. We may assume that 
\[\ell = |\C{B}'| = \binom{x}{n-r-1}\] for some real number $x \ge n-r-1$. It follows from Theorem~\ref{lov} that 
\[|\partial^{(r-1)}\C{B}'| \ge \binom{x}{r-1}.\] Let us put $r = (1/2 - \beta) n$ and $x = \vartheta n$. We now calculate, ignoring error terms that are $o(1)$, what values $\beta$ and $\vartheta$ can take. We know that $\eps \le \beta \le 1/2 -\eps$. Since $x \ge n-r-1$, we also know that $\vartheta \ge 1/2 + \beta$. On the other hand, since 
\[ \binom{\vartheta n}{(1/2 + \beta)n} = \ell \le \NN^{1-\eta} = \binom{n-1}{r-1}^{1-\eta} \le \binom{n}{r}^{1-\eta}= \binom{n}{(1/2 - \beta)n}^{1-\eta},\]
it follows from Stirling's approximation for the factorial function that there exists some $\delta'(\eps, \eta)>0$ such that $\vartheta \le 1-\delta'$.

Hence, it suffices to check that there exists a $\delta = \delta(\eps, \eta) > 0$ for which the inequality
\[ \binom{\vartheta n}{ (1/2 -\beta) n} \ge \binom{\vartheta n}{(1/2+\beta) n }^{1 + \delta}\]
holds for all $\beta \in [\eps, 1/2 - \eps]$ and $\vartheta \in [1/2 + \beta, 1-\delta']$ as long as $n$ is sufficiently large. This is easily checked using Stirling's formula. 
\end{proof}

\section{Proof of the main result}\label{s:proof}
Armed with Lemmas~\ref{set-pairs} and~\ref{kk}, we are now ready to prove Theorem~\ref{t:main}.

\begin{proof}[Proof of Theorem~\ref{t:main}]
Let us fix $\eps>0$ and assume that $r \le (1/2 - \eps) n$. Clearly, it is enough to prove Theorem~\ref{t:main} for all sufficiently small $\eps$; it will be convenient to assume that $\eps < 1/10$. As mentioned earlier, Bollob\'as, Narayanan and Raigorodskii have proved Theorem~\ref{t:main} in a much stronger form when $r = o(n^{1/3})$. So to avoid having to distinguish too many cases, we shall assume that $r$ grows with $n$; for concreteness, let us suppose that $r \ge n^{1/4}$. A consequence of these assumptions is that in this range, $\VV$, $\NN$ and $\MM$ all grow much faster than any polynomial in $n$.

First, let $Y$ denote the (random) number of independent sets $\C{A}\subset \NR$ in $K_{p}(n,r)$ with $|\C{A}| = \NN+1$ and $\ell(\C{A}) = 1$; in other words, independent sets of size $\NN+1$ which contain an entire star. We begin by showing that there exists a $c' = c'(\eps)$ such that if $p \le \NN^{-c'}$, then $Y>0$ with high probability. Clearly,
\[\E[Y] =  \binom{n}{1}\binom{\VV-\NN}{1}(1-p)^{\MM}.\]
Note that if $r \le (1/2 - \eps)n$, then we may choose a suitably small $c' = c'(\eps)$ such that $\MM \ge \NN^{c'}$. It follows that if $c'$ is sufficiently small, then
\[
\E[Y] \ge  n(\VV-\NN)\exp{\left(-(p + p^2)\MM\right)} \ge  (e+o(1)) n(\VV-\NN),
\]
so $\E[Y] \to \infty$ when $p \le \NN^{-c'}$. 

Therefore, to show that $Y>0$ with high probability, it suffices to show that $\V [Y] = o(\E[Y]^2)$ or equivalently, that $\E[(Y)_2] = (1+o(1))\E[Y]^2$, where $\E[(Y)_2] = \E[ Y(Y-1)]$ is the second factorial moment of $Y$.

Writing $\C{S}_x$ for the star centred at $x$, we note that 
\[
\E[(Y)_2] = \sum_{a,b,A,B} \P \left( \C{S}_a \cup \{A\} \, \mbox{and}\, \C{S}_b \cup \{B\} \,\mbox{are independent} \right),\]
the sum being over ordered $4$-tuples $(a,b,A,B)$ with $a,b \in [n]$, $A \in \NR \setminus \C{S}_a$ and $B \in \NR \setminus \C{S}_b$ such that $(a,A) \neq (b, B)$. Now, observe that 
\begin{align*}
\sum_{a \neq b} \P \left( \C{S}_a \cup \{A\} \, \mbox{and}\, \C{S}_b \cup \{B\} \,\mbox{are independent} \right) &\le (n^2) (\VV-\NN)^2 (1-p)^{(2-o(1))\MM}\\
&= (1+o(1))\E[Y]^2,
\end{align*}
and
\begin{align*}
\sum_{a = b, A \neq B} \P \left( \C{S}_a \cup \{A\} \, \mbox{and}\, \C{S}_b \cup \{B\} \,\mbox{are independent} \right) &\le n(\VV-\NN)^2 (1-p)^{2\MM}\\
&= o(\E[Y]^2).
\end{align*}

By Chebyshev's inequality, we conclude that $Y>0$ with high probability, so the independence number of $K_p(n,r)$ is at least $\NN+1$ with high probability if $p \le \NN^{-c'}$.

Next, for each $\ell \ge 1$, let $X_\ell$ denote the (random) number of independent sets $\C{A}\subset \NR$ in $K_{p}(n,r)$ with $|\C{A}| = \NN$ and $\ell(\C{A}) = \ell$. To complete the proof of Theorem~\ref{t:main}, it clearly suffices to show that for some $c = c(\eps)>0$, all of the $X_\ell$ are zero with high probability provided $p \ge \NN^{-c}$. We shall prove this by distinguishing three cases depending on which of Theorem~\ref{HM}, Lemma~\ref{set-pairs} and Lemma~\ref{kk} is to be used.

Let $C = C(\eps)$ be as in Theorem~\ref{friedgut}. Note that since $r \le (1/2 - \eps)n$, it is easy to check using Stirling's approximation that we can choose positive constants $c_m=c_m(\eps)$ and $c_r=c_r(\eps)$ such that $\MM \ge \NN^{c_m}$ and $\RR \le \NN^{1-c_r}$. 

We now set $L_m = \NN^{c_m}/2$ and $L_r = \NN^{1-c_r/4}$ and distinguish the following three cases.

\textbf{Case 1: $\ell \le L_m$ .}
Let $\C{A} \subset \NR$ be a family of cardinality $\NN$ with $\ell(\C{A}) = \ell$. Since 
\[\ell \le L_m = \NN^{c_m}/2 \le \MM - 2,\] 
we see that $\C{A}^*$, the largest intersecting subfamily of $\C{A}$, satisfies 
\[ |\C{A}^*| = \NN - \ell \ge \NN - \MM + 2.\] 
It follows from Theorem~\ref{HM} that there is an $x \in [n]$ for which $\C{A}^*$ is contained in the star centred at $x$. Consider the $\ell$ sets in $\C{A}\setminus \C{A}^*$. Any such set is disjoint from exactly $\MM$ members of the star centred at $x$ and hence from at least $\MM - \ell$ members of $\C{A}^*$. This tells us that $e(\C{A}) \ge \ell(\MM - \ell)$. Since $\ell \le \MM/2$, it follows that 
\begin{align*}
\E[X_\ell] &\le n\binom{\NN }{\ell}\binom{\VV}{\ell}(1-p)^{\ell(\MM - \ell)}\\
&\le  n\binom{2^n }{\ell}^2\exp(-p\ell\MM /2)\\
&\le \exp(2n\ell - p\ell\MM/2).
\end{align*}

Hence, if $c \le c_m/2$ so that $p \ge \NN^{-c_m/2}$, it is clear that
\[ \sum_{\ell = 1}^{L_m} \E[X_\ell] \le \sum_{\ell = 1}^{L_m} \exp\left(2n\ell - \frac{\ell \NN^{c_m/2}}{2}\right) = o(1).\]
So with high probability, for each $1 \le \ell \le L_m$, the random variable $X_\ell$ is zero.

\textbf{Case 2: $\ell \ge L_r$.} Again, let $\C{A} \subset \NR$ be a family of cardinality $\NN$ with $\ell(\C{A}) = \ell$. We know from Lemma~\ref{set-pairs} that
\[ e(\C{A}) \ge \frac{\ell^2}{2\RR} \ge \frac{\NN^{2-c_r/2}}{2 \NN^{1-c_r}} = \frac{\NN^{1+c_r/2}}{2}.\]
So it follows that
\[ \sum_{l\ge L_r} \E[X_\ell] \le \binom{\VV}{\NN}\exp\left(-p\frac{\NN^{1+c_r/2}}{2}\right) \le \exp\left(n\NN -p\frac{\NN^{1+c_r/2}}{2}\right).\]
Hence, if $c \le c_r/4$ so that $p \ge \NN^{-c_r/4}$, we have
\[ \sum_{l \ge L_r} \E[X_\ell] \le \exp\left(n\NN - \frac{\NN^{1+c_r/4}}{2}\right) = o(1).\]
So once again, with high probability, the sum $\sum_{\ell \ge L_r} X_\ell$ is zero.

Before we proceed further, let us first show that that we may now assume without loss of generality that $r \ge \eps n$. This is because one can check that the arguments in Cases 1 and 2 together prove Theorem~\ref{t:main} when $r \le \eps n$ for all sufficiently small $\eps$. It is easy to check using Stirling's formula that if $\eps$ is sufficiently small, indeed, if $\eps < 1/10$ for example, then it is possible to choose positive constants $c'_m(\eps)$ and $c'_r(\eps)$ so that for all $r \le \eps n$, we have $\MM \ge \NN^{c'_m}$, $\RR \le \NN^{1-c'_r}$ and $\NN^{c'_m}/2 \ge \NN^{1-c'_r/4}$. So the arguments above yield a proof of Theorem~\ref{t:main} when $r \le \eps n$. Therefore, in the following, we assume that $r \ge \eps n$.

\textbf{Case 3: $L_m \le \ell \le L_r$.}
As before, consider any family $\C{A} \subset \NR$ of cardinality $\NN$ with $\ell(\C{A}) = \ell$. First note that since $\eps n \le r \le (1/2 -\eps n)$ and $\ell \le L_r = \NN^{1-c_r/4}$ where $c_r$ is a constant depending only on $\eps$, by Lemma~\ref{kk}, there exists a $\delta = \delta(\eps)$ such that 
\[e(\C{A}) \ge \ell^{1+\delta} - C\ell. \]
Since $\ell \ge L_m = \NN^{c_m}/2$, it follows that
\[e(\C{A}) \ge \ell^{1+\delta} - C\ell \ge \ell^{1+\delta/2}\]
for all sufficiently large $n$.

Next, consider $\C{A}^*$, the largest intersecting subfamily of $\C{A}$, which has cardinality $\NN - \ell$. We know from Theorem~\ref{friedgut} that there exists an $x \in [n]$ such that $|\C{A}^*_x| \ge \NN - C\ell$, so $|\C{A}_x| \ge \NN - C\ell$. It is then easy to see that
\begin{align*} 
\E[X_\ell] &\le n \binom{\NN}{C\ell}\binom{\VV}{C\ell}(1-p)^{\ell^{1+\delta/2}}\\
&\le \exp\left(\ell\left(2Cn - p\ell^{\delta/2}\right)\right).  
\end{align*}

Hence, if $c \le c_m\delta/4$ so that $p \ge \NN^{-c_m\delta/4}$, it follows that 
\[ \sum_{\ell = L_m}^{L_r} \E[X_\ell] \le \sum_{\ell = L_m}^{L_r} \exp\left(\ell\left(2Cn - \NN^{c_m\delta/4}/2\right)\right) = o(1), \]
so with high probability, for each $L_m \le \ell \le L_r$, the random variable $X_\ell$ is zero.

Putting the different parts of our argument together, we find that if $0 < \eps < 1/10$,
\[
c = c(\eps) = \min \left(\frac{c_m(\eps)}{2},\frac{c'_m(\eps)}{2},\frac{c_r(\eps)}{4},\frac{c'_r(\eps)}{4},\frac{c_m(\eps)\delta(\eps)}{2} \right)
\] 
and $p \ge \NN^{-c}$, then for all $r = r(n) \le (1/2 -\eps)n$, we have
\[
\P\left(\alpha\left(K_p(n,r)\right) = \binom{n-1}{r-1}\right) \to 1
\]
as $n \to \infty$. This completes the proof of Theorem~\ref{t:main}. 
\end{proof}

\section{Refinements}\label{s:cont}
We briefly discuss how one might tighten up the arguments in Theorem~\ref{t:main} so as to improve the dependence of $c(\eps)$ on $\eps$ in the result. However, since it seems unlikely to us that these methods will be sufficient to determine the precise critical threshold at which Theorem~\ref{t:main} ceases to hold, we shall keep the discussion in this section largely informal.

\subsection{Containers for sparse sets in the Kneser graph}
The first approach we sketch involves using ideas from the theory of `graph containers' to count large sparse sets in the Kneser graph more efficiently. 

The theory of graph containers was originally developed to efficiently count the number of independent sets in a graph satisfying some kind of `supersaturation' condition. The basic principle used to construct containers for graphs can be traced back to the work of Kleitman and Winston~\citep{kleit}. A great deal of work has since gone into refining and generalising their ideas, culminating in the results of Balogh, Morris and Samotij~\citep{cont1} and Saxton and Thomason~\citep{cont2}; these papers also give a detailed account of the history behind these ideas and we refer the interested reader there for details about how the general methodology was developed. Here we shall content ourselves with a brief discussion of how these ideas might be used to improve the dependence of $c(\eps)$ on $\eps$ in Theorem~\ref{t:main}.

Let us write $Y_m=Y_m(n,r)$ for the number of families $\C{A} \subset \NR$ with $|\C{A}| = \NN$ and $e(\C{A})=m$. Clearly, to show that $\alpha(K_p(n,r)) = \NN$ with high probability, it suffices to show that $\sum_{m\ge1} Y_m (1-p)^m = o(1)$. Hence, it would be useful to have good estimates for $Y_m$. We shall derive some bounds for $Y_m$; see Theorem~\ref{babycont} below. These bounds are not strong enough (especially for small values of $m$) to prove Theorem~\ref{t:main}. However, note that in our proof of Theorem~\ref{t:main}, we use the somewhat cavalier bound of $\binom{\VV}{\NN}$ for the number of families $\C{A}$ of size $\NN$ for which $\ell(\C{A})$ is equal to some prescribed value (in Case~2 of the proof); we can instead use Theorem~\ref{babycont} to count more efficiently.

To prove an effective container theorem, one needs to first establish a suitable supersaturation property. Lov\'asz~\citep{shannon} determined the second largest eigenvalue of the Kneser graph; by combining Lov\'asz's result with the expander mixing lemma, Balogh, Das, Delcout, Liu and Sharifzadeh~\citep{typical} proved the following supersaturation theorem for the Kneser graph.

\begin{proposition}\label{supersat}
Let $n,r,k\in\N$ and suppose that $n > 2r$ and $k \le \VV - \NN$. If $\C{A} \subset \NR$ has cardinality $\NN + k$, then $e(\C{A}) \ge k\MM/2$.\qed
\end{proposition}
Using Proposition~\ref{supersat}, we prove the following container theorem for the Kneser graph.

\begin{theorem}\label{babycont}
For every $\eps > 0$, there exists a ${\hat C} = {\hat C}(\eps)>0$ such that for every $\beta > 0$ and all $n, r, m \in \N$ with $\eps n \le r \le (1/2 - \eps)n$, the following holds: writing 
\[k_1 = {\hat C} \left(\frac{\NN}{\beta \MM} +  \left(\frac{m\NN }{\beta \MM} \right)^{1/2}\right) \]
and 
\[ k_2 = k_1 + {\hat C} \beta\NN ,
\]
there exist, for $1 \le i \le \sum_{j=0}^{k_1}\binom{\VV}{j}$, families $\C{B}_i \subset \NR$ each of cardinality at most $\NN + k_2$ with the property that each $\C{A} \subset \NR$ with $e(\C{A}) \le m$ is contained in one of these families.
\end{theorem}

The advantage of this formulation of Theorem~\ref{babycont} in terms of $k_1$, $k_2$ and $\beta$  is that we can apply the theorem with a value of $\beta>0$ suitably chosen for the application at hand. 

It is easy to check from Theorem~\ref{babycont} that $Y_m=Y_m(n,r)$, the number of families $\C{A} \subset \NR$ with $|\C{A}| = \NN$ and $e(\C{A})=m$, satisfies
\begin{align*}
Y_m(n,r) &\le \left(\sum_{j=0}^{k_1}\binom{\VV}{j}\right) \binom{\NN + k_2}{\NN} = 2\binom{\VV}{k_1} \binom{\NN+k_2}{k_2} \le 2\binom{\VV}{k_1} \binom{\VV}{k_2}\\
&\le 2\exp\left({\hat C}n \left(\beta\NN + \frac{2\NN}{\beta \MM} + \left(\frac{4m\NN }{\beta \MM} \right)^{1/2}\right)\right)\end{align*}  
for all $\beta > 0$ such that $k_1 < \VV/3$. We can then optimise this bound by choosing $\beta$ depending on how large $m$ is in comparison to $\MM$ and $\NN$. For example, when $m \ge \NN / \MM^{1/2}$, we can take $\beta = (m/\NN \MM)^{1/3}$ and easily check that $ Y_m(n,r) \le \exp (10{\hat C}n(m\NN^2/\MM)^{1/3})$. The reader may check that this estimate for $Y_m$ when combined with the Hilton--Milner theorem is sufficient to prove Theorem~\ref{t:main} when $r/n$ is bounded above by and away from $\vartheta$, where $\vartheta \approx 0.362$ is, writing $H(x) = - x \log x -(1-x)\log(1-x)$, the unique real solution to the equation
\[ 3(1-\vartheta)H\left(\frac{\vartheta}{1-\vartheta}\right) = 2H(\vartheta)\]
in the interval $(0,1)$.

\begin{proof}[Proof of Theorem~\ref{babycont}]
We start by proving a lemma whose proof is loosely based on the methods of Saxton and Thomason~\citep{cont2}. Before we state the lemma, let us have some notation. Given a graph $G=(V,E)$ and $U \subset V(G)$, we write 
\[ \mu (U) = \frac{|E(G[U])|}{|V|}; \] 
in other words, $\mu(U)$ is the number of edges induced by $U$ divided by the number of vertices of $G$. Also, we write $\C{P}(X)$ for the collection of all subsets of a set $X$.

\begin{lemma}\label{ST}
Let $G=(V,E)$ be a graph with average degree $d$ and maximum degree $\Delta$. For every $a \ge 0$ and $b > 0$, there is a map $\Scr{C} : \C{P}(V)\to \C{P}(V)$ with the following property: for every $U \subset V$ with $\mu(U) \le a$, there is a subset $T \subset V$ such that
\begin{enumerate}
\item $T \subset U \subset \Scr{C}(T)$,
\item $|T| \le 2|V|(a/b d)^{1/2} + |V|/b d$, and 
\item $\mu(\Scr{C}(T))\le 2\Delta(a/b d)^{1/2} + \Delta/b d + b d$.
\end{enumerate}
\end{lemma}
\begin{proof}
We shall describe an algorithm that constructs $T$ given $U$. The algorithm will also construct $\Scr{C}(T)$ in parallel; it will be clear from the algorithm that $\Scr{C}(T)$ is entirely determined by $T$ and in no way depends on $U$.

Fix a linear ordering of the vertex set $V$ of $G$. If $u$ and $v$ are adjacent and $u$ precedes $v$ in our ordering, we call $v$ a \emph{forward neighbour} of $u$ and $u$ a \emph{backward neighbour} of $v$. For a vertex $v \in V$, we write $F(v)$ for the set of its forward neighbours. 

We begin by setting $T = \emptyset$ and $A = V$. We shall iterate through $V$ in the order we have fixed and add vertices to $T$ and remove vertices from $A$ as we go along; at any stage, we write $\Gamma(T)$ to denote the set of those vertices which, at that stage, have $k$ or more backward neighbours in $T$ where $k$ is the least integer strictly greater than $(abd)^{1/2}$.

As we iterate through the vertices of $V$ in order, we do the following when considering a vertex $v$.
\begin{enumerate}
\item\label{one} If $v \in \Gamma(T)$, we remove $v$ from $A$; if it is also the case that $v \in U$, then we add $v$ to $T$. 
\item If $v \notin \Gamma(T)$, we consider the size of $S = F(v) \setminus \Gamma(T)$. 
\begin{enumerate}
\item\label{two} If $|S| \ge b d$, we remove $v$ from $A$; if it is also the case that $v \in U$, then we add $v$ to $T$. 
\item If $|S| < b d$, we do nothing.
\end{enumerate}
\end{enumerate}

The algorithm outputs $T$ and $A$ when it terminates; we then set $\Scr{C}(T) = A \cup T$. It is clear from the algorithm that $\Scr{C}(T)$ is uniquely determined by $T$ and that $T \subset U \subset \Scr{C}(T)$.

We first show that $|T| \le 2|V|(a/b d)^{1/2} + |V|/b d$. Consider the partition $T = T_1 \cup T_2$ where $T_1$ consists of those vertices which were added to $T$ on account of condition~(\ref{one}) and $T_2$ of those vertices which were added to $T$ when considering condition~(\ref{two}). The upper bound for $|T|$ follows from the following two claims.

\begin{claim}
$|T_1| \le |E(G[U])|/k$.
\end{claim}
\begin{proof}
Clearly, each vertex of $T_1$ has at least $k$ backward neighbours in $T \subset U$. Hence, $k|T_1| \le  |E(G[U])|$.
\end{proof}

\begin{claim}
$|T_2| \le k|V|/b d$.
\end{claim}
\begin{proof}
Let us mark all the edges from $v$ to $F(v) \setminus \Gamma(T)$ when a vertex $v$ gets added to $T$ on account of condition~(\ref{two}). The number of marked edges is clearly at least $bd|T_2|$. On the other hand, by the definition of $\Gamma(T)$, each vertex is joined to at most $k$ of its backward neighbours by a marked edge. Hence, $bd|T_2| \le k|V|$.
\end{proof}

Consequently, since $(abd)^{1/2} < k \le (abd)^{1/2} + 1$, we have
\begin{align*}
|T| &= |T_1| + |T_2| \le \frac{a|V|}{k} + \frac{k|V|}{bd} \\
&\le \frac{a|V|}{(abd)^{1/2}} + \frac{((abd)^{1/2} + 1)|V|}{bd} \le 2|V|\left(\frac{a}{b d}\right)^{1/2} + \frac{|V|}{b d}.
\end{align*}
It remains to show that $\mu(\Scr{C})\le 2\Delta(a/b d)^{1/2} + \Delta/b d + b d$. To see this, recall that $\Scr{C}(T) = A \cup T$ and notice that 
\[ |E(G[\Scr{C}(T)])| \le \Delta|T| + |E(G[A])| \le \Delta|T| + bd|V|. \]
To see the last inequality, i.e., $|E(G[A])| \le bd|V|$, note that a vertex $v$ is removed from $A$ by our algorithm unless we have $|F(v) \setminus \Gamma(T)| < bd$ at the stage where we consider $v$. Since each member of $\Gamma(T)$ is (eventually) removed from $A$, we see that each vertex of $A$ has at most $bd$ forward neighbours in $A$ and the inequality follows. The claimed bound for $\mu(\Scr{C})$ then follows from our previously established upper bound for $|T|$.
\end{proof}

To prove Theorem~\ref{babycont}, we now combine Lemma~\ref{ST} with Proposition~\ref{supersat}. First note that the Kneser graph $K(n,r)$ has $\VV = n\NN/r$ vertices and is $(n-r)\MM/r$ regular. 

Let us take ${\hat C}(\eps) = 20/\eps^2$. It is easy to check that given $\beta > 0$ and a family $\C{A} \subset \NR$ with $e(\C{A}) \le m$, we can apply Lemma~\ref{ST} with $a = m/\VV$ and $b = \beta$ to get families $\C{T} \subset \NR$ and $\Scr{C}(\C{T}) \subset \NR$ such that $\C{T} \subset \C{A} \subset \Scr{C}(\C{T})$, $|\C{T}| \le k_1$ and $e(\Scr{C}(\C{T})) \le  k_2 \MM/2$. Hence, by Proposition~\ref{supersat}, we see that $|\Scr{C}(\C{T})| \le \NN +  k_2$. The theorem then follows by taking the families $\Scr{C}(\C{T})$ for every $\C{T} \subset \NR$ with $|\C{T}| \le k_1$.
\end{proof}

\subsection{Stability for the Kruskal--Katona theorem}
An important ingredient in our proof of Theorem~\ref{t:main} is Lemma~\ref{kk} which gives a uniform lower bound, using Theorem~\ref{friedgut} and the Kruskal--Katona theorem, for $e(\C{A})$ in terms of $\ell(\C{A})$ when the size of $\C{A}$ is large. 

However, there is a price to be paid for proving such a uniform bound: the bound is quite poor for most families to which the lemma can be applied. Indeed, the families which are extremal for the argument in the proof of Lemma~\ref{kk} must possess a great deal of structure. Instead of the Kruskal--Katona theorem, one should be able to use a stability version of the Kruskal--Katona theorem, as proved by Keevash~\citep{Keevash} for example, to prove a more general result that accounts for the structure of the family under consideration.

\section{Conclusion}\label{s:conc}
Several problems related to the question considered here remain. First of all, it would be good to determine the largest possible value of $c(\eps)$ with which Theorem~\ref{t:main} holds. It is likely that one needs new ideas to resolve this problem. 

Second, one would also like to know what happens when $r$ is very close to $n/2$. Perhaps most interesting is the case when $n = 2r+1$; one would like to know the values of $p$ for which we have $\alpha(K_p(2r+1,r))=\binom{2r}{r-1}$ with high probability. A simple calculation shows that $p = 3/4$ is the threshold at which we are likely to find a star and an $r$-set not in the star all the edges between which are missing in $K_p(2r+1,r)$ which suggests that the critical threshold should be $3/4$. However, it would even be interesting to show that $\alpha(K_p(2r+1,r))=\binom{2r}{r-1}$ with high probability for, say, all $p \ge 0.999$.

\section*{Acknowledgements}
The first author is partially supported by a Simons fellowship, NSF CAREER grant DMS-0745185, Arnold O. Beckman Research Award (UIUC Campus Research Board 13039) and Marie Curie grant FP7-PEOPLE-2012-IIF 327763. The second author would like to acknowledge support from EU MULTIPLEX grant 317532 and NSF grant DMS-1301614. 

Some of the research in this paper was carried out while the authors were visitors at the Alfr\'ed R\'enyi Institute of Mathematics. This research was continued while the third author was a visitor at the University of Memphis. The authors are grateful for the hospitality of the R\'enyi Institute and the third author is additionally grateful for the hospitality of the University of Memphis.

We would also like to thank Andrew Thomason for some helpful discussions about graph containers.

\bibliographystyle{amsplain}
\bibliography{ekr_transference}

\end{document}